\documentclass[9.0pt, a4paper]{article}
\usepackage{amsmath}
\usepackage{fancyhdr}
\usepackage{color}
\usepackage{amsmath}
\usepackage{amsfonts}
\usepackage{dsfont}
\usepackage{mathrsfs}
\usepackage{bbm}
\usepackage{latexsym}
\usepackage{graphicx}
\usepackage{epstopdf}
\usepackage{amssymb}
\usepackage{indentfirst}
\usepackage{subfigure}
\usepackage{booktabs}
\usepackage{authblk}
\usepackage{cite}
\usepackage{float}
\usepackage[numbers,sort&compress]{natbib}
\usepackage{caption}
\usepackage{booktabs}
\usepackage{threeparttable}
\usepackage{multicol}
\usepackage{multirow}
\usepackage{geometry}
\usepackage{lineno}
\usepackage{amsbsy}
\usepackage{bm}
\usepackage{dcolumn}
\usepackage{mathrsfs}
\usepackage[title]{appendix}
\usepackage{stmaryrd}
\usepackage{tikz}
\tikzset{>=stealth}

\newtheorem{remark}{Remark}[section]
\newtheorem{theorem}{Theorem}[section]
\newtheorem{lemma}[theorem]{Lemma}

\newtheorem{assumption}{Assumption}
\numberwithin{equation}{section}

\def\qed{\hfill$\Box$\vspace{8pt}}

\begin{document}
\title{More than one Author with different Affiliations}
\author[1]{Qigang Liang}
\author[1,2]{Xuejun Xu}
\author[3]{Shangyou Zhang}

\affil[1]{\small School of Mathematical Science, Tongji University, Shanghai 200092, China, qigang$\_$liang@tongji.edu.cn}
\affil[2]{\small Institute of Computational Mathematics, Academy of Mathematics and Systems Science, Chinese Academy of Sciences, Beijing 100190, China, xxj@lsec.cc.ac.cn}
\affil[3]{\small Department of Mathematical Sciences, University of Delaware, Newark, DE 19716, USA, szhang@udel.edu}
\title{On a Sharp Estimate of Overlapping Schwarz Methods in $\bm{H}(\bm{{\rm curl}};\Omega)$ and $\bm{H}({\rm div};\Omega)$ }\date{}
\maketitle

{\bf{Abstract}:}\ \ The previous proved-bound is $C(1+\frac{H^2}{\delta^2})$ 
   for the condition number of the 
   overlapping domain decomposition  $\bm{H}(\bm{{\rm curl}};\Omega)$ 
   and $\bm{H}({\rm div};\Omega)$ methods, where $H$ and $\delta$ are the sizes of
   subdomains and overlaps respectively. 
 But all numerical results indicate that the best bound is $C(1+\frac{H}{\delta})$.
  In this work, we solve this long-standing open problem by proving that
  $C(1+\frac{H}{\delta})$ is indeed the best bound.

{\bf{Keywords}:}\ \ Maxwell equations, H-curl elements, H-div elements, Helmholtz decomposition, overlapping domain decomposition. \hspace*{2pt}

\section{Introduction}

\par Overlapping Schwarz method is one of the most important methods for computing the large-scale discrete problems arising from partial differential equations (PDE). This domain decomposition (DD) method is essentially parallel and has been extensively studied in the literature (see, e.g., \cite{MR1273155,MR1771050,MR1348039,MR1376107,MR1794350,MR3647952,MR3985469,MR1838270,MR4379971} and the references therein). Generally speaking, the iterative convergence rate (e.g., PCG, preconditioned GMRES) depends on the condition number of the discrete system.  Therefore, it is very important to obtain the sharp estimate of the preconditioned algebraic systems resulting from overlapping Schwarz methods.

\par For a long time, the best bound of the condition number of the overlapping domain decomposition method is  $C\left(1+ \frac{H^{2}}{\delta^{2}} \right)$ for the second order elliptic boundary value problems, cf. \cite{MRDryjaWidlund,MR2104179}, where $\delta$ is the overlapping size and $H$ is the diameter of subdomains. In 1994, Dryja and Widlund \cite{MR1273155} first improved the bound to $C\left(1+ \frac{H}{\delta} \right)$. Subsequently, Brenner \cite{MR1771050} proved the best bound is $C\left(1+\frac{H}{\delta}\right)$. The same techniques used are applied to the fourth order elliptic boundary value problems and high-frequency Helmholtz problems (see \cite{MR1771050,MR1348039,MR1376107,MR3647952}).

\par The same situation happens to the analysis of the two-level overlapping domain decomposition method in $\bm{H}(\bm{{\rm curl}};\Omega)$.
Toselli \cite{MR1794350} proved an upper bound of  $C\left(1+\frac{H^{2}}{\delta^{2}}\right)$ for the condition number. For many years, people wonder if the best bound should be $C\left(1+\frac{H}{\delta}\right)$. Numerical results \cite{MR1794350} indicate that the best bound should be $C\left(1+\frac{H}{\delta}\right)$. Bonazzoli et al. \cite{MR3985469} posed this open problem if the best bound is $C\left(1+ \frac{H}{\delta} \right)$. In this paper, we close this open problem and prove that $C\left(1+ \frac{H}{\delta} \right)$ is the best bound. 

\par The key ideas in obtaining the sharp estimate of overlapping Schwarz methods in $\bm{H}(\bm{{\rm curl}};\Omega)$ and $\bm{H}({\rm div};\Omega)$ are as follows.  Firstly, the functions are limited to one element of the coarse triangulation where all functions are $\bm{H}^1$ locally.  This way get sharper estimates than the results in \cite{MR1794350}. Secondly, by the Helmholtz decomposition, we get a solenoidal subspace in $\bm{H}(\bm{{\rm curl}};\Omega)$ and an irrotational subspace in $\bm{H}({\rm div};\Omega)$ which is also in $\bm{H}^{1}(\Omega)$ when $\Omega$ is convex.  After the decomposition we can utilize the techniques from the domain decomposition method for $H^1$ problems.

\par The rest of this paper is organized as follows: In section 2, we introduce model problems and some preliminaries. We introduce the overlapping Schwarz method and give a stable spacial decomposition for $\bm{H}(\bm{{\rm curl}};\Omega)$-elliptic problems in Section 3. An extension to $\bm{H}({\rm div};\Omega)$-elliptic problems is introduced in Section 4. Finally, we present a conclusion in Section 5.

\section{Model problems and preliminaries}
\par Throughout this paper, we use standard notations for Sobolev spaces $H^{m}(D)$ and $H_{0}^{m}(D)$ with their associated norms $\|\cdot\|_{m,D}$ and semi-norms $|\cdot|_{m,D}$. We denote by $L^{2}(D):=H^{0}(D)$, and use $(\cdot,\cdot)_{0,D}$ to represents the $L^{2}$-inner product and $||\cdot||_{0,D}$ represents the corresponding $L^{2}$-norm. For vector field space, we use bold font $\bm{L}^{2}(D)$ and $\bm{H}^{m}(D)$ to represent $[L^{2}(D)]^{d}$ and $[H^{m}(D)]^{d}$, respectively ($d=2,3$), and still use the notations of the norms $\|\cdot\|_{m,D}$, $|\cdot|_{m,D}$ and $||\cdot||_{0,D}$ with its inner product $(\cdot,\cdot)_{0,D}$. If $D=\Omega$, we drop the subscript $\Omega$ in associated norms or semi-norms or inner products. Let
\begin{align*}
\bm{H}(\bm{{\rm curl}};\Omega):=\{\bm{u}\in \bm{L}^{2}(\Omega)\ |\ \bm{{\rm curl}}\bm{u}\in \bm{L}^{2}(\Omega)\ \}
\end{align*}
equipped with the norm $||\cdot||^{2}_{\bm{{\rm curl}}}=||\cdot||_{0}^{2}+||\bm{{\rm curl}}\cdot||_{0}^{2}$. $\bm{H}_{0}(\bm{{\rm curl}};\Omega)$ represents a subspace where the vector valued functions have a vanishing tangential trace at domain boundary. $\bm{H}(\bm{{\rm curl}}_{0};\Omega)$ represents another supspace where the vector valued functions have a vanishing $\bm{{\rm curl}}$.  We also let
\begin{align*}
\bm{H}({\rm div};\Omega):=\{\bm{u}\in \bm{L}^{2}(\Omega)\ |\ {\rm div}\bm{u}\in \bm{L}^{2}(\Omega)\  \}
\end{align*}
equipped with the norm $||\cdot||^{2}_{{\rm div}}=||\cdot||_{0}^{2}+||{\rm div}\cdot||_{0}^{2}$.  Let $\bm{H}({\rm div}_{0};\Omega)$ be a subspace where the vector valued functions are divergence-free. For convenience of notations in this paper, we define 
\begin{align}\label{H0p}
\bm{H}^{\perp}_{0}(\bm{{\rm curl}};\Omega):=\bm{H}_{0}(\bm{{\rm curl}};\Omega)\cap \bm{H}({\rm div}_{0};\Omega).
\end{align}
\par In this paper, we assume $\Omega$ is convex, bounded and simple connected. Then we know that the kernel of $\bm{{\rm curl}}$ operator in $\bm{H}_{0}(\bm{{\rm curl}};\Omega)$ is $\nabla{H_{0}^{1}(\Omega)}$. Moreover, the Helmholtz decomposition holds
\begin{align*}
\bm{H}_{0}(\bm{{\rm curl}};\Omega)=\nabla{H_{0}^{1}(\Omega)}\oplus \bm{H}^{\perp}_{0}(\bm{{\rm curl}};\Omega),
\end{align*}
where $\oplus$ is an orthogonal decomposition under $\bm L^2(\Omega)$.
For the theoretical analysis in the following, we define an operator (called as Hodge operator in \cite{MR2009375}) for the above decomposition,
\begin{align} \label{T-p} \begin{aligned}
    &\Theta^{\perp}  \  : \  \bm{H}_{0}(\bm{{\rm curl}};\Omega)\to \bm{H}^{\perp}_{0}(\bm{{\rm curl}};\Omega), \\
    &\Theta^{\perp} (\bm w) = \Theta^{\perp} (\nabla s+ \bm w^\perp ) = \bm w^\perp, 
   \end{aligned}
\end{align} where $s\in H_0^1(\Omega)$ and $\bm w^\perp \in \bm{H}^{\perp}_{0}(\bm{{\rm curl}};\Omega)$.  Therefore, $\Theta^\perp$ is ${\bf curl}$-preserving that
\begin{align} \label{T-preserving}  
   {\bf curl} \Theta^{\perp} \bm w &={\bf curl} \bm w .
\end{align} 
Using the Sobolev embedding theorem (see, e.g., \cite{MR1626990,MR2009375,MR851383}), we known that
\begin{align}\label{Align_Embedding}
\bm{H}^{\perp}_{0}(\bm{{\rm curl}};\Omega)=\bm{H}_{0}(\bm{{\rm curl}};\Omega)\cap \bm{H}({\rm div}_{0};\Omega)\hookrightarrow \bm{H}^{1}(\Omega).
\end{align}

\par Consider the model problem 
\begin{equation}\label{Equation_curlcurl}
     \begin{cases}
        \bm{{\rm curl}}\;\bm{{\rm curl}}\bm{u}+\bm{u}=\bm{f} \ \ &\text{in $\Omega,$}\\
         \ \ \ \ \ \bm{n}\times\bm{u}=\bm{0}\ \ &\text{on $\partial\Omega,$}
     \end{cases}
\end{equation}
where  $\bm{n}$ represents the outward unit normal vector on $\partial{\Omega}$. The variational form for \eqref{Equation_curlcurl} is as follows: 
\begin{equation}\label{Equation_curlcurl_variation}
    \begin{cases}
    \text{Given $\bm{f} \in \bm{L}^{2}(\Omega)$, find $\bm{u}\in \bm{H}_{0}(\bm{{\rm curl}};\Omega)$ such that}\\
    \text{$a_{\bm{{\rm curl}}}(\bm{u},\bm{v})=(\bm{f},\bm{v})_{0}\ \ \ \ \forall\ \bm{v}\in \bm{H}_{0}(\bm{{\rm curl}};\Omega)$},
    \end{cases}
\end{equation}
where 
\begin{equation}\label{a-curl}
  a_{\bm{{\rm curl}}}(\bm{w},\bm{v}):=\int_{\Omega}\big(\bm{{\rm curl}}\bm{w}\cdot\bm{{\rm curl}}\bm{v}+\bm{w}\cdot\bm{v}\big)dx
\end{equation}
for all $\bm{w},\bm{v}\in \bm{H}(\bm{{\rm curl}};\Omega)$. By the Lax-Milgram theorem, it is easy to see that the solutions with respect to \eqref{Equation_curlcurl_variation} are well-posed.

\par Let $\mathcal{T}_{h}$ be a shape-regular and quasi-uniform triangulation. We  consider the $k$-th N\'ed\'elec element in $\bm{H}_{0}(\bm{{\rm curl}};\Omega)$ as follows
\begin{equation}\label{ND-h}
\mathcal{ND}_{h,0}:=\{\bm{v}\in \bm{H}_{0}(\bm{{\rm curl}};\Omega)\ |\ \bm{v}|_{\tau}\in \mathcal{ND}_k (\tau),\ \forall\ \tau \in \mathcal{T}_{h}\ \},
\end{equation}
where $\mathcal{ND}_k(\tau)$ is the $k$-th order local N\'ed\'elec space on element $\tau$ (see, e.g., \cite{MR592160,MR864305,MR2009375}). 
The discrete variational form of \eqref{Equation_curlcurl_variation} may be written as:
\begin{equation}
             \begin{cases}\label{Equation_curlcurl_discrete_variation}
             \text{Given $\bm{f}\in \bm{L}^{2}(\Omega)$, find $\bm{u}_{h}\in\mathcal{ND}_{h,0}$ such that }\\
             a_{\bm{{\rm curl}}}(\bm{u}_{h},\bm{v}_{h})= (\bm{f},\bm{v}_{h})_{0}\ \ \ \ \forall\ \bm{v}_{h}\in \mathcal{ND}_{h,0}.
             \end{cases}
\end{equation}
Define an operator $A_{h}:\mathcal{ND}_{h,0}\to \mathcal{ND}_{h,0}$ such that $(A_{h}\bm{w}_{h},\bm{v}_{h})=a_{\bm{{\rm curl}}}(\bm{w}_{h},\bm{v}_{h})$ for all $\bm{w}_{h},\bm{v}_{h}\in \mathcal{ND}_{h,0}.$ We denote the discrete divergence-free space by 
\begin{align}\label{ND-p} 
\mathcal{ND}_{h,0}^{\perp}=\{\bm{v}_{h}\in \mathcal{ND}_{h,0}\ |\ (\bm{v}_{h},\nabla{p}_{h})_{0}=0,\ \ \forall\ p_{h}\in S_{h,0}\ \},
\end{align}
with $S_{h,0}$ being a continuous and piecewise $P_{k+1}$ polynomial space on $\mathcal{T}_h$ with vanishing trace on $\partial{\Omega}$. Therefore, we have the discrete Helmholtz decomposition 
\begin{align} \label{D-decompose}
\mathcal{ND}_{h,0}=\nabla{S_{h,0}}\oplus \mathcal{ND}_{h,0}^{\perp}.
\end{align}
It is easy to see that $\mathcal{ND}_{h,0}^{\perp}\not\subset
  \bm{H}^{\perp}_{0}(\bm{{\rm curl}};\Omega)$, where
    $ \bm{H}^{\perp}_{0}(\bm{{\rm curl}};\Omega)$
is defined in \eqref{H0p}. 
Let a subspace be
\begin{align*} \bm{V}^{+}:=\Theta^{\perp}\mathcal{ND}_{h,0}^{\perp}\subset \bm{H}^{\perp}_{0}(\bm{{\rm curl}};\Omega),
\end{align*} where $\Theta^\perp$ is defined in \eqref{T-p}.   Define another operator 
\begin{align*}
P_{h}:\bm{H}^{\perp}_{0}(\bm{{\rm curl}};\Omega)\to \bm{V}^{+} 
\end{align*} such that
\begin{align}\label{P-h}
(\bm{{\rm curl}}P_{h}\bm{w},\bm{{\rm curl}}\bm{v})_{0}=(\bm{{\rm curl}}\bm{w},\bm{{\rm curl}}\bm{v})_{0}\ \ \ \ \forall\ \bm{w}\in \bm{H}^{\perp}_{0}(\bm{{\rm curl}};\Omega),\ \bm{v}\in \bm{V}^{+}.
\end{align}
Due to the fact that the Poincar\'e inequality holds in $\bm{H}^{\perp}_{0}(\bm{{\rm curl}};\Omega)$, we know that the operator $P_{h}$ is well-defined in \eqref{P-h}. Further, we extend the operator $P_{h}$ to $\bm{H}_{0}(\bm{{\rm curl}};\Omega)$ by
\begin{align} \label{P-h-1} 
      \begin{aligned}
        &P_{h} \nabla  s  :=0, \quad s\in H^1_0(\Omega), \\
        &P_{h}\bm w =P_h (\nabla  s+ \bm w^\perp) = P_h \bm w^\perp, \quad  \ \bm w^\perp \in \bm{H}_{0}^\perp(\bm{{\rm curl}};\Omega).
     \end{aligned}
\end{align} 
It holds that for any $\bm{w}_{h}^{\perp}\in \mathcal{ND}_{h,0}^{\perp}$, we have
\begin{align}\label{P-T}
P_{h}\bm{w}_{h}^{\perp}&= P_h(\nabla s+  \Theta^{\perp} \bm{w}_{h}^{\perp})
       =   P_h \Theta^{\perp} \bm{w}_{h}^{\perp} =  \Theta^{\perp} \bm{w}_{h}^{\perp}
\end{align} for some $s\in H^1_0(\Omega)$.
By \eqref{T-preserving}, we have 
\begin{align}\label{P-e}
   \bm{{\rm curl}} \; P_{h}\bm{w}_{h}^{\perp} = \bm{{\rm curl}} \;   \Theta^{\perp} \bm{w}_h^\perp &=    \bm{{\rm curl}} \;\bm{w}_h^\perp .
\end{align}
\par The following lemma holds (see Lemma 10.6 in \cite{MR2104179}). 
\begin{lemma}\label{Lemma_Ph}
Let $\Omega$ be convex. Then the following error estimate holds,
\begin{equation}\notag
||\bm{u}_{h}^{\perp}-P_{h}\bm{u}_{h}^{\perp}||_{0}\leq Ch||\bm{{\rm curl}}\bm{u}^{\perp}_{h}||_{0}\ \ \ \ \forall\ \bm{u}^{\perp}_{h}\in \mathcal{ND}_{h,0}^{\perp},
\end{equation}
with $C$ independent of $h$ and $\bm{u}^{\perp}_{h}$,
where $P_h$ is defined in \eqref{P-h},\ \eqref{P-h-1} and $\mathcal{ND}_{h,0}^{\perp}$
   is defined in \eqref{ND-p}.
\end{lemma}

\section{Overlapping Schwarz methods in $\bm{H}(\bm{{\rm curl}};\Omega)$}

\par Let $\mathcal{T}_{H}:=\{K_{i}\}_{i=1}^{N}$ be a shape-regular and quasi-uniform coarse triangular or tetrahedral mesh on $\Omega$, where $H:=\max\{H_{i}\ |\ i= 1,2,...,N\}$.
We let domain be subdivided into $N$ subdomains where
\begin{align*} \Omega_i = K_i, \quad i=1, \dots, N. 
\end{align*} The fine shape-regular and quasi-uniform triangulation $\{\tau\}$ is obtained by subdividing $\mathcal{T}_{H}$ and we denote it by $\mathcal{T}_{h}=\{\tau\}$. We may construct the edge element spaces $\mathcal{ND}_{H,0}\subset \mathcal{ND}_{h,0}$ on $\mathcal{T}_{H}$ and $\mathcal{T}_{h}$ but it is well-known that $\mathcal{ND}_{h,0}^{\perp}\not\subset \mathcal{ND}_{h,0}^{\perp}$. To get overlapping subdomains $(\Omega_{i}^{'},\ 1\leq i\leq N)$, we enlarge a subdomain $\Omega_{i}$ by adding a size-$\delta$ layer of fine elements, where $\delta =O({\rm dist}(\partial\Omega_{i}\setminus\partial\Omega,\partial\Omega_{i}^{'}\setminus\partial\Omega))$.
We define, see the gray region in Figure \ref{Figure_triangular_mesh} in 2D (The 3D case is same), 
   the overlapping region inside $\Omega_i'$ by
\begin{align} \label{domains} \begin{aligned} 
\Omega_{i,j,\delta}&=\bigcup_{\tau\in \mathcal{T}_h,\ \tau\subset (\Omega_i'\cap\Omega_j') } \tau,\\
\Omega_{i,\delta }:&=\bigcup_{\Omega_j'\cap \Omega_i' \ne \emptyset}
    \Omega_{i,j,\delta}. \end{aligned}
\end{align}

\begin{figure}[H]
\centering
\begin{tikzpicture}[scale=1]
\draw [step=2] (0,0) grid (6,6);
\draw [black] (0,6) -- (6,0);
\draw [black] (0,4) -- (4,0);
\draw [black] (0,2) -- (2,0);
\draw [black] (2,6) -- (6,2);
\draw [black] (4,6) -- (6,4);

\pgfmathsetmacro{\mysqrt}{sqrt(2)}
\filldraw [gray!20] (1.7,2) -- (1.7,4.3) -- (2,4+0.3*\mysqrt) -- (4+0.3*\mysqrt, 2) -- (4.3,1.7) -- (2,1.7) -- (1.7,2);
\filldraw [gray!20,thick] (2,2) -- (2,4) -- (4,2) -- (2,2);
\filldraw [white,thick] (2.3,2.3) -- (2.3,3.7-0.3*\mysqrt) -- (3.7-0.3*\mysqrt,2.3) -- (2.3,2.3);

\draw [black,thick] (2,2) -- (2,4) -- (4,2) -- (2,2);
\draw [red,thick,dashed] (1.7,2) -- (1.7,4.3) -- (2,4+0.3*\mysqrt) -- (4+0.3*\mysqrt, 2) -- (4.3,1.7) -- (2,1.7) -- (1.7,2);

\draw [red,thick,dashed] (2.3,2.3) -- (2.3,3.7-0.3*\mysqrt) -- (3.7-0.3*\mysqrt,2.3) -- (2.3,2.3);

\draw [red,thick,dashed,->] (6.3,2.5) to [in = -5, out = 200] (4.2,2.3);
\node at (6.3,2.5) [right] {\small $\Omega_i'$};

\draw [thick,->] (6.3,3.5) to [in = -5, out = 200] (3.35,2.7);
\node at (6.3,3.5) [right] {\small $\Omega_{i}$};

\draw [blue,thick,dashed,->] (6.3,4.4) to [in = -5, out = -200] (3.1,3.1);
\node at (6.3,4.4) [right] {\small $\Omega_{i,\delta}$};
\end{tikzpicture}
\caption{The diagrammatic presentation of $\Omega_{i}^{'},\Omega_{i},\Omega_{i,\delta}$ for triangular mesh}
\label{Figure_triangular_mesh}
\end{figure}
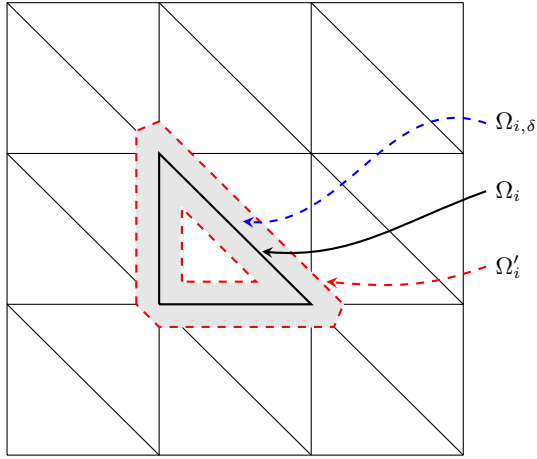

\par We decompose the finite element space $\mathcal{ND}_{h,0}$ in \eqref{ND-h} into overlapping subspaces,
\begin{align*}
   \bm{V}_{i}:=\mathcal{ND}_{h,0}\cap\bm{H}_{0}(\bm{\mbox{curl}};\Omega_{i}^{'}),\ \ \ \ i=1,2,...,N . 
\end{align*}
For describing the overlapping domain decomposition preconditioner for $A_{h}$, we define $A_{H}:\mathcal{ND}_{H,0}\to \mathcal{ND}_{H,0}$ such that
\begin{align*}  
(A_{H}\bm{w}_{H},\bm{v}_{H})_{0}=a_{\bm{{\rm curl}}}(\bm{w}_{H},\bm{v}_{H})\ \ \ \ \forall\ \bm{w}_{H},\bm{v}_{H}\in \mathcal{ND}_{H,0}.
\end{align*}
Similarly, we also define $A_{i}:\bm{V}_{i}\to \bm{V}_{i}$ such that
\begin{align*}
(A_{i}\bm{w}_{i},\bm{v}_{i})_{0}=a_{\bm{{\rm curl}}}(\bm{w}_{i},\bm{v}_{i})\ \ \ \ \forall\ \bm{w}_{i},\bm{v}_{i}\in \bm{V}_{i}.
\end{align*}
We denote by $Q_{H}:\bm{L}^{2}(\Omega)\to \mathcal{ND}_{H,0}$ a $\bm{L}^{2}$-orthogonal projector and $Q_{i}:\bm{L}^{2}(\Omega)\to \bm{V}_{i},\ (i=1,2,...,N)$ $\bm{L}^{2}$-orthogonal projectors. So the preconditioner is
\begin{align}\label{B-h-1}
B_{h}^{-1}=A_{H}^{-1}Q_{H}+\sum_{i=1}^{N}A_{i}^{-1}Q_{i}.
\end{align}

\par Next, we introduce an assumption on overlapping domain decomposition (see \cite{MR2104179}). 
\begin{assumption}\label{Assumption}
The partition $\{\Omega_{i}^{'}\}_{i=1}^{N}$ may be colored using at most $N_{0}$ colors, in such a way that subdomains with the same color are disjoint. The integer $N_{0}$ is independent of $N$.
\end{assumption}

\par According to Assumption \ref{Assumption}, we obtain a partition of unity, and then there exists a family continuous and piecewise linear polynomials
$\{\theta_{i}\}_{i=1}^{N}$, which satisfy the following properties
\begin{align}\label{theta}\begin{aligned}
&{\rm supp}(\theta_{i})\subset \overline{\Omega_{i}^{'}},\ \ \ \ 0\leq \theta_{i}\leq 1,\\
&\sum_{i=1}^{N}\theta_{i}\equiv 1,\ \ \ \ x\in \Omega,\\
&\big(\nabla{\theta_{i}}\big)|_{\Omega_{i}^{\circ}}=0,\ \ \ \ ||\nabla{\theta_{i}}||_{0,\infty,\Omega_{i,\delta }}\leq \frac{C}{\delta }, \end{aligned}
\end{align}
where $\Omega_{i}^{\circ}=\Omega_{i}^{'}\backslash\overline{\Omega_{i,\delta}}$.

\par Next, we first present the main result in this paper, and delay its proof. 

\begin{theorem}\label{Theorem_Mainresult}
Let Assumption \ref{Assumption} hold. Then for any $\bm{v}_{h}\in \mathcal{ND}_{h,0}$, we have
\begin{align}\label{Align_Mainresult}
\frac{1}{C_{1}(1+ {H}/{\delta})} a_{\bm{{\rm curl}}}(\bm{v}_{h},\bm{v}_{h})\leq a_{\bm{{\rm curl}}}(B_{h}^{-1}A_{h}\bm{v}_{h},\bm{v}_{h})\leq C_{2} a_{\bm{{\rm curl}}}(\bm{v}_{h},\bm{v}_{h}),
\end{align}
with the constants $C_{1}$ and $C_{2}$ independent of $h,\ H,\ \delta$ and $\bm{v}_{h}$, but not $N_{0}$, where $B_h^{-1}$ is defined in \eqref{B-h-1} and $N_0$ is defined in 
   Assumption \ref{Assumption}.
\end{theorem}

\begin{remark}
By the Assumption \ref{Assumption},  the upper bound in \eqref{Align_Mainresult} is standard,
  cf. \cite{MR2104179}. We will prove the lower bound in \eqref{Align_Mainresult}.
\end{remark}

\par For convenience of theoretical analysis, using the ``local" argument (see \cite{MR1066830}), we may denote by $Q_{0,\Omega_{i}}:\bm{L}^{2}(\Omega_{i})\to \bm{P}_{0}(\Omega_{i}),\ \ \Omega_{i}\in \mathcal{T}_{H}$, a local $\bm{L}^{2}$-orthogonal projector.  
 We have
\begin{align}\label{Align_Q0i_Approximation}
||\bm{w}-Q_{0,\Omega_{i}}\bm{w}||_{0,\Omega_{i}}\leq CH|\bm{w}|_{1,\Omega_{i}} \ \ \ \ \forall\ \bm{w}\in \bm{H}^{1}(\Omega_{i}).
\end{align}

\par In order to prove the main result, we first give some technical lemmas. In the following theoretical analysis, we also take advantage of the global $\bm{L}^{2}$-orthogonal projector $Q_{H}$. The following lemma holds (see \cite{MR1794350,MR2104179}).

\begin{lemma}\label{Lemma_QH}
Let $\mathcal{T}_{H}$ be shape-regular and quasi-uniform. Then for $\bm{u}\in \bm{H}^{1}(\Omega)$, we have
\begin{align}\label{Q-h} 
    \begin{aligned}
||\bm{{\rm curl}}(Q_{H}\bm{u})||_{0}&\leq C|\bm{u}|_{1},\\
||\bm{u}-Q_{H}\bm{u}||_{0}&\leq CH|\bm{u}|_{1},
    \end{aligned}
\end{align}
with the constant $C$ independent of $\bm{u}$ and $H$, where $Q_H$ is the $\bm{L}^2$-orthogonal projection to $\mathcal{ND}_{H,0}$.
\end{lemma}

\begin{lemma}\label{Lemma_Inequality_of_InnerStri}
Let $\bm{w}$ be a piecewise $H^1$ function, i.e. $\bm w|_{\Omega_{i}}\in \bm{H}^{1}(\Omega_{i})$ on each $\Omega_{i}$. It holds that
\begin{align} \label{w-b}
  ||\bm{w}||_{0,\Omega_{i_0,\delta} }^{2}
\leq C\delta ^{2}\sum_{j=0}^{I_{0}}
    \left\{\big(1+\frac{H}{\delta }\big)|\bm{w}|_{1,\Omega_{i_j}}^{2}
+\frac{1}{\delta H}||\bm{w}||^{2}_{0,\Omega_{i_j}}\right\},
\end{align}
where $\Omega_{i_0,\delta}$ is the layer of small elements around the boundary
   of $\Omega_{i_{0}}'$, defined in \eqref{domains},
  and $\Omega_{i_1}, \dots, \Omega_{i_{I_0}}$ are the $I_0$ neighbor subdomains which
   have nonempty intersection 
   $\Omega_{i_j}'\cap \Omega_{i_0}'$ in the definition \eqref{domains} of $\Omega_{i_0,\delta}$.
\end{lemma}

\begin{proof}
A similar proof is given in \cite{MR2104179} for Lemma 3.10 there, except we have a
  piecewise $\bm H^1$ function while it is global $\bm H^1$ in  \cite{MR2104179}. 
For simplicity, we illustrate the proof in 2D case. The 3D case is same. 

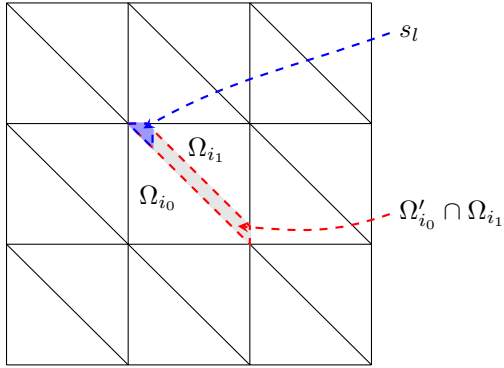
\begin{figure}[H]
\centering
\begin{tikzpicture}[scale=0.8]
\draw [step=2] (0,0) grid (6,6);
\draw [black] (0,6) -- (6,0);
\draw [black] (0,4) -- (4,0);
\draw [black] (0,2) -- (2,0);
\draw [black] (2,6) -- (6,2);
\draw [black] (4,6) -- (6,4);

\pgfmathsetmacro{\mysqrt}{sqrt(2)}
\filldraw [gray!20,thick] (2,4)--(2.3,4)--(4,2.3)--(4,2)--(2,4);
\filldraw [blue!40,thick]  (2,4)--(2.3,4)--(2.4,3.9)--(2.4,3.6)--(2,4);

\draw [red,thick,dashed]   (2,4)--(2.3,4)--(4,2.3)--(4,2)--(2,4);
\draw [blue,thick,dashed]  (2,4)--(2.3,4)--(2.4,3.9)--(2.4,3.6)--(2,4);

\node at (2.5,2.4) [above] { $\Omega_{i_0}$};
\node at (3.3,3.2) [above] { $\Omega_{i_{1}}$};

\draw [red,thick,dashed,->] (6.3,2.5) to [in = -5, out = 200] (3.8,2.3);
\node at (6.3,2.5) [right] { $\Omega_{i_0}' \cap \Omega_{i_{1}}$};

\draw [blue,thick,dashed,->] (6.3,5.5) to [in = 30, out =200] (2.25,3.92);
\node at (6.3,5.5) [right] { $s_{l}$};

\end{tikzpicture}
\caption{The definition of $s_l$ and the stripe $\Omega_{i_0}' \cap \Omega_{i_{1}}$.}
\label{Figure_triangular_mesh_1}
\end{figure}

We claim all triangles in $\Omega_{i_0,\delta}$ belong to at least one of the following
  stripes, cf. Figures \ref{Figure_triangular_mesh_1} and \ref{Figure_triangular_mesh_2}
   where $I_0=12$,
\begin{align}\label{list-s} \begin{aligned}
    \Omega'_{i_1} \cap \Omega_{i_0}, \ \;
    \Omega_{i_5}'\cap \Omega_{i_0}, \ \;
         \Omega_{i_9}'\cap \Omega_{i_0}, \ \;
    \Omega_{i_1}'\cap \Omega_{i_2}, \ \dots, \ \;
           \Omega_{i_{I_0}}'\cap \Omega_{i_1}, \\
   \Omega_{i_1}\cap \Omega_{i_0}', \ \;
   \Omega_{i_5} \cap \Omega_{i_0}', \ \;
         \Omega_{i_9} \cap \Omega_{i_0}', \ \;
    \Omega_{i_1} \cap \Omega_{i_2}', \ \dots,  \ \;
           \Omega_{i_{I_0}} \cap \Omega_{i_1}'.  \end{aligned}
\end{align} We note  that inside each stripe $\bm w$ is $\bm H^1$.
We will prove \eqref{w-b} on one stripe first, then get \eqref{w-b} by summing over
   all strips in \eqref{list-s}.

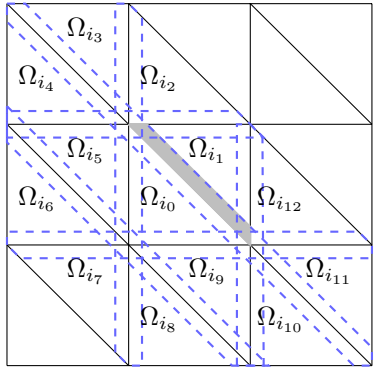
\begin{figure}[H]
\centering
\begin{tikzpicture}[scale=0.8]
\draw [step=2] (0,0) grid (6,6);
\draw [black] (0,6) -- (6,0);
\draw [black] (0,4) -- (4,0);
\draw [black] (0,2) -- (2,0);
\draw [black] (2,6) -- (6,2);
\draw [black] (4,6) -- (6,4);

\pgfmathsetmacro{\mysqrt}{sqrt(2)}

\filldraw [gray!50,thick] (2,4)--(2.3,4)--(4,2.3)--(4,2)--(2,4);

\draw [blue!60,thick,dashed]   (0,6)--  (0.3,6)--(6,0.3)--(6,0)--(5.7,0)--(0,5.7)--(0,6);
\draw [blue!60,thick,dashed]   (0,4.3)--(4.3,0)--(3.7,0)--(0,3.7)--(0,4.3);
\draw [blue!60,thick,dashed]   (1.78,0.22)--(1.78,6)--(2,6)--(2.22,5.78)--(2.22,0)--(2,0)--(1.78,0.22);
\draw [blue!60,thick,dashed]   (0,2)--(0,2.22)--(5.78,2.22)--(6,2)--(6,1.78)--(0.22,1.78)--(0,2);
\draw [blue!60,thick,dashed]   (3.78,0)--(3.78,4)--(4,4)--(4.2,3.78)--(4.22,0)--(4,0)--(3.78,0);
\draw [blue!60,thick,dashed]   (0,3.78)--(0,4.22)--(3.78,4.22)--(4.22,3.78)--(0,3.78);


\node at (2.5,2.4) [above] { $\Omega_{i_{0}}$};
\node at (3.3,3.2) [above] { $\Omega_{i_{1}}$};
\node at (2.5,4.4) [above] { $\Omega_{i_{2}}$};
\node at (1.3,5.2) [above] { $\Omega_{i_{3}}$};
\node at (0.5,4.4) [above] { $\Omega_{i_{4}}$};
\node at (1.3,3.2) [above] { $\Omega_{i_{5}}$};
\node at (0.5,2.4) [above] { $\Omega_{i_{6}}$};
\node at (1.3,1.2) [above] { $\Omega_{i_{7}}$};
\node at (2.5,0.4) [above] { $\Omega_{i_{8}}$};
\node at (3.3,1.2) [above] { $\Omega_{i_{9}}$};
\node at (4.5,0.4) [above] { $\Omega_{i_{10}}$};
\node at (5.3,1.2) [above] { $\Omega_{i_{11}}$};
\node at (4.5,2.4) [above] { $\Omega_{i_{12}}$};
\end{tikzpicture}

\caption{All $\Omega_{i_j}$ related to the estimation on $\Omega_{i_{0},\delta}$}
\label{Figure_triangular_mesh_2}
\end{figure}

\par We separate one stripe $\Omega_{i_0}'\cap \Omega_{i_1}$ into finite patches $\{s_l\}_{l=1}^{n_{i_{0}}}$
     of triangles $\tau_{j}$ from $\mathcal{T}_h$,  cf. Figure \ref{Figure_triangular_mesh_1},
\begin{align*}  
s_l = \cup_{j=1}^{m_{l}} \tau_{j},  \quad \ \ |s_l| = C\delta^2, \quad |\partial s_l \cap \partial \Omega_{i_{1}} | = C\delta.  
\end{align*}
By using the Ponicar\'e-Friedrichs inequality on $s_{l}$, because $\bm{w}\in \bm{H}^{1}(\Omega_{i})$,  we have
\begin{align*}
||\bm{w}||_{0,s_{l}}^{2}\leq C\delta ^{2}|\bm{w}|_{1,s_{l}}^{2}+C\delta ||\bm{w}||^{2}_{0,\partial{s_{l}}\cap \partial{\Omega_{i_{1}}}}.
\end{align*}
Summing over the patches $\{s_{l}\}_{1\leq l\leq n_{i_0}}$,
    we get, using the trace theorem in $\Omega_{i_{1}}$
\begin{align*}
||\bm{w}||_{0,\Omega_{i_0}'\cap \Omega_{i_1} }^{2}
&\leq C\delta ^{2}|\bm{w}|_{1,\Omega_{i_1}}^{2}
+C\delta ||\bm{w}||^{2}_{0,\partial{\Omega_{i_0}}\cap \partial{\Omega_{i_{1}}}}\\ 
&\leq C\delta ^{2}|\bm{w}|_{1,\Omega_{i_1}}^{2}
    +C\delta H|\bm{w}|^{2}_{1,\Omega_{i_{1}}}
+C\delta H^{-1}||\bm{w}||^{2}_{0,\Omega_{i_{1}}}\\ 
&\leq C\delta ^{2}\{\big(1+\frac{H}{\delta }\big)|\bm{w}|_{1,\Omega_{i_{1}}}^{2}
+\frac{1}{\delta H}||\bm{w}||^{2}_{0,\Omega_{i_{1}}}\}.
\end{align*}

Summing over all strips $\Omega_{i_j}'\cap \Omega_{i_k}$
   in \eqref{list-s}, by finite covering, we get \eqref{w-b}.
 This completes the proof of the lemma. \qed
\end{proof}

For any $\bm{u}_{h}\in \mathcal{ND}_{h,0}$, we decompose it as 
    $\bm{u}_{h}=  \sum_{i=0}^N \bm{u}_{i} $, where
\begin{align}\label{d-u} \begin{aligned}
\bm{u}_{0} &:=\nabla{q_{0}}+\bm{w}_{0}\in \mathcal{ND}_{H,0},\\ 
 \bm{u}_{i}&:=\nabla{q_{i}}+\bm{w}_{i}\in \bm{V}_{i},\ \ \ \ i=1,2,...,N,\\
  q_{0} &= \tilde I_H q_h, \\
  q_i &= I_h (\theta_i(q_h-q_0)), \\
  \bm{w}_{0} &=Q_H P_h \bm w_h^\perp, \\
  \bm{w}_{i} &= \Pi_{E,h} (\theta_i (\bm w_h^\perp - \bm w_0)),  \\
  \bm{u}_{h}&= \nabla q_h + \bm w_h^\perp,
\end{aligned} 
\end{align} 
where $q_h$ and $\bm w_h^\perp$ are defined in \eqref{D-decompose}, 
 $\Pi_{E,h}$ is the H-curl interpolation operator to $\mathcal{ND}_{h,0}$,
   $\theta_i$ is defined in \eqref{theta}, $P_h$ is defined in \eqref{P-h}, 
  $Q_H$ is defined in \eqref{Q-h}, $I_h$ is the nodal value interpolation to 
  $S_{h,0}$ defined in \eqref{ND-p} and $\tilde I_H$ is the Scott-Zhang interpolation
   operator (see \cite{MR1011446}) to $S_{H,0}$. In \eqref{d-u}, the operation $\Pi_{E,h} (\theta_i (\bm w_h^\perp - \bm w_0))$ is well-defined, because $\theta_i (\bm w_h^\perp - \bm w_0)\in \bm{H}_{0}(\bm{{\rm curl}};\Omega)$ and $\bm w_0\in \mathcal{ND}_{H,0}$.
   We check the decomposition,
   \begin{align*}
   \sum_{i=0}^{N}\bm{u}_{i}
   &=\nabla{q_{0}}+\bm{w}_{0}+\sum_{i=1}^{N}(\nabla{q_{i}}+\bm{w}_{i})\\
   &=\nabla{q_{0}}+\bm{w}_{0}+\sum_{i=1}^{N}\left(\nabla{I_h (\theta_i(q_h-q_0))}+\Pi_{E,h} (\theta_i (\bm w_h^\perp - \bm w_0))\right)\\
   &=\nabla{q_{0}}+\bm{w}_{0}+\nabla{I_h \left(\sum_{i=1}^{N}\theta_i(q_h-q_0)\right)}
   +\Pi_{E,h} \left(\sum_{i=1}^{N}\theta_i (\bm w_h^\perp - \bm w_0)\right)\\
   &=\nabla{q_{0}}+\bm{w}_{0}+\nabla{I_h \left(q_h-q_0\right)}
   +\Pi_{E,h} \left(\bm w_h^\perp-\bm w_0\right)\\
   &=\nabla{q_{0}}+\bm{w}_{0}+\nabla{\left(q_h-q_0\right)}
   +\left(\bm w_h^\perp-\bm w_0\right)\\
   &= \nabla q_h + \bm w_h^\perp=\bm{u}_{h}.
   \end{align*}
  
\par In order to prove the lower bound in \eqref{Align_Mainresult}, we only need to give a stable decomposition (see Chapter 2 in \cite{MR2104179}), and then prove that the stable parameter can be bounded by $C(1+\frac{H}{\delta})$.  
\begin{theorem}\label{Theorem_Stable_Decomposition}
Let $\bm{u}_{h}\in \mathcal{ND}_{h,0}$ be decomposed in \eqref{d-u}.
It holds that 
\begin{align}\label{upper}
  \sum_{i=0}^{N}a_{\bm{{\rm curl}}}(\bm{u}_{i},\bm{u}_{i})
\leq C(1+\frac{H}{\delta})a_{\bm{{\rm curl}}}(\bm{u}_{h},\bm{u}_{h}). 
\end{align}
\end{theorem}

\begin{proof}
By the discrete Helmholtz decomposition in the last equation in \eqref{d-u},
we decompose $\bm{u}_{h}$ in \eqref{d-u} as 
\begin{align*} 
\bm{u}_{h}=\nabla{q_0}+\bm{w}_{0} + \sum_{i=1}^N (\nabla{q_i}+\bm{w}_{i}).
\end{align*}

\par For the terms $\nabla{q_{i}}$ of $\bm u_h$, using the result for continuous finite element spaces conforming in $H_{0}^{1}(\Omega)$ given in the proof of lemma 3.12 in \cite{MR2104179}
and \eqref{d-u}, we have
\begin{align}\label{Align_Stable_Decomposition_kercurl}
\begin{aligned}
\sum_{i=0}^{N}a_{\bm{{\rm curl}}}(\nabla{q_{i}},\nabla{q_{i}})
&=\sum_{i=0}^{N}|q_{i}|^{2}_{1,\Omega_{i}^{'}}
 \leq C(1+\frac{H}{\delta})|q_{h}|^{2}_{1}\\
&=C(1+\frac{H}{\delta}) a_{\bm{{\rm curl}}}(\nabla{q_{h}},\nabla{q_{h}}).
\end{aligned}
\end{align}

\par For the terms $\bm{w}_{i}$ of $\bm u_h$,  we will prove the following by the four steps below. 
\begin{align}\label{Align_Stable_Decomposition_perp}
\begin{aligned}
\sum_{i=0}^{N}a_{\bm{{\rm curl}}}(\bm{w}_{i},\bm{w}_{i})
&=||\bm{{\rm curl}}\bm{w}_{0}||_{0}^{2}+||\bm{w}_{0}||_{0}^{2}
+\sum_{i=1}^{N}||\bm{{\rm curl}}\bm{w}_{i}||_{0,\Omega_{i}^{'}}^{2}+\sum_{i=1}^{N}||\bm{w}_{i}||_{0,\Omega_{i}^{'}}^{2}\\
&\leq C(1+\frac{H}{\delta})a_{\bm{{\rm curl}}}(\bm{w}^{\perp}_{h},\bm{w}^{\perp}_{h}).
\end{aligned}
\end{align}
\begin{itemize}
\item[(1)] For the term $||\bm{{\rm curl}}\bm{w}_{0}||_{0}$ in \eqref{Align_Stable_Decomposition_perp}, by \eqref{d-u}, \eqref{Q-h},
   \eqref{Align_Embedding} 
     and \eqref{P-e}, we get
\begin{align}\label{Align_Coarse_Component_a1}
||\bm{{\rm curl}}\bm{w}_{0}||_{0}
=||\bm{{\rm curl}}Q_{H}P_{h}\bm{w}_{h}^{\perp}||_{0}
\leq C|P_{h}\bm{w}_{h}^{\perp}|_{1}
\leq C||\bm{{\rm curl}}P_{h}\bm{w}_{h}^{\perp}||_{0}
= C||\bm{{\rm curl}}\bm{w}_{h}^{\perp}||_{0}.
\end{align}
\item[(2)] For the term $||\bm{w}_{0}||_{0}$ in \eqref{Align_Stable_Decomposition_perp}, by \eqref{d-u}, $\bm{L}^2$-orthogonal projector $Q_H$ and \eqref{P-T}, we have
\begin{align}\label{Align_Coarse_Component_a2}
||\bm{w}_{0}||_{0}
=||Q_{H}P_{h}\bm{w}_{h}^{\perp}||_{0}
\leq ||P_{h}\bm{w}_{h}^{\perp}||_{0}
=||\Theta^{\perp}\bm{w}_{h}^{\perp}||_{0}
\leq  ||\bm{w}_{h}^{\perp}||_{0}.
\end{align} Here the last inequality follows the argument
\begin{align*}
  ( \Theta^{\perp}\bm{w}_{h}^{\perp} , \bm g)_0
   = ( \Theta^{\perp}(\nabla s+ \tilde {\bm{w}}^{\perp}) , \bm g)_0
   = ( \tilde {\bm{w}}^{\perp} , \bm g)_0
   = ( \nabla s+ \tilde {\bm{w}}^{\perp} , \bm g)_0 
   = (  \bm{w}_h^{\perp} , \bm g)_0 
   \end{align*} 
with $\bm g=\Theta^{\perp}\bm{w}_{h}^{\perp}$ and the Cauchy-Schwarz inequality.
   
\item[(3)] For the terms $||\bm{{\rm curl}}\bm{w}_{i}||_{0,\Omega_{i}^{'}}$ in \eqref{Align_Stable_Decomposition_perp}, by the properties of $\{\theta_{i}\}_{i=1}^{N}$ and $\Pi_{E,h}$ (see Lemma 10.8 in \cite{MR2104179}), we have
\begin{align*}
\sum_{i=1}^{N}||\bm{{\rm curl}}\bm{w}_{i}||_{0,\Omega_{i}^{'}}^{2}
\leq C\sum_{i=1}^{N}\delta ^{-2}||\bm{v}||_{0,\Omega_{i,\delta }}^{2}+C||\bm{{\rm curl}}\bm{v}||_{0}^{2},
\end{align*}
where \begin{align*} \bm v = \bm{w}_h^\perp -\bm w_0=  \bm{w}_h^\perp 
   - Q_{H}P_{h}\bm{w}_{h}^{\perp}.  
\end{align*}
Denoting by $\widetilde{\bm{v}}:=P_{h}\bm{w}_{h}^{\perp}-Q_{H}P_{h}\bm{w}_{h}^{\perp}$ and using the triangle inequality, we obtain from above inequality that 
\begin{align} \label{Align_Local_Estimates_I1I2I3}
    \begin{aligned}
    \sum_{i=1}^{N}||\bm{{\rm curl}}\bm{w}_{i}||_{0,\Omega_{i}^{'}}^{2}
    &\leq C\sum_{i=1}^{N}\delta ^{-2}||\widetilde{\bm{v}}||_{0,\Omega_{i,\delta }}^{2}+ C\sum_{i=1}^{N}\delta ^{-2}||\widetilde{\bm{v}}-\bm{v}||_{0,\Omega_{i,\delta }}^{2}+C||\bm{{\rm curl}}\bm{v}||_{0}^{2}\\
    &=:I_{1}+I_{2}+I_{3}.
    \end{aligned}
\end{align}

For the first term $I_{1}$ in \eqref{Align_Local_Estimates_I1I2I3}, by  Lemma \ref{Lemma_Inequality_of_InnerStri} and the fact that $\widetilde{\bm{v}}\in \bm{H}^{1}(\Omega_{i})$, we get, because of finite overlapping,
\begin{align}\label{Align_widetildebmv_estimate3}
I_{1}=C\sum_{i=1}^{N}\delta ^{-2}||\widetilde{\bm{v}}||_{0,\Omega_{i,\delta }}^{2}
&\leq C(1+\frac{H}{\delta})\sum_{i=1}^{N}|\widetilde{\bm{v}}|^{2}_{1,\Omega_{i}}+ C\frac{1}{\delta H}||\widetilde{\bm{v}}||^{2}_{0,\Omega}.
\end{align}
For the first term in \eqref{Align_widetildebmv_estimate3}, 
   by the triangle inequality, the inverse estimate
    and \eqref{Align_Q0i_Approximation}, we obtain
\begin{align*}
|\widetilde{\bm{v}}|_{1,\Omega_{i}} 
&= |P_{h}\bm{w}_{h}^{\perp}-Q_{H}P_{h}\bm{w}_{h}^{\perp}|_{1,\Omega_i}  
 \leq |P_{h}\bm{w}_{h}^{\perp}|_{1,\Omega_{i}}
+|Q_{H}P_{h}\bm{w}_{h}^{\perp} |_{1,\Omega_{i}}\\
&= |P_{h}\bm{w}_{h}^{\perp}|_{1,\Omega_{i}}
+|Q_{H}P_{h}\bm{w}_{h}^{\perp}-Q_{0,\Omega_{i}}P_{h}\bm{w}_{h}^{\perp}|_{1,\Omega_{i}}\\
&\leq |P_{h}\bm{w}_{h}^{\perp}|_{1,\Omega_{i}}
   +CH^{-1}||Q_{H}P_{h}\bm{w}_{h}^{\perp}-Q_{0,\Omega_{i}}P_{h}\bm{w}_{h}^{\perp}||_{0,\Omega_{i}}\\
&\leq |P_{h}\bm{w}_{h}^{\perp}|_{1,\Omega_{i}}
    +CH^{-1}|| Q_{H}P_{h}\bm{w}_{h}^{\perp} - P_{h}\bm{w}_{h}^{\perp}||_{0,\Omega_{i}}\\
&\ \ \ \ +CH^{-1}||P_{h}\bm{w}_{h}^{\perp} 
     -Q_{0,\Omega_{i}}P_{h}\bm{w}_{h}^{\perp}
        ||_{0,\Omega_{i}}\\
&\leq |P_{h}\bm{w}_{h}^{\perp}|_{1,\Omega_{i}}+CH^{-1}||P_{h}\bm{w}_{h}^{\perp}-Q_{H}P_{h}\bm{w}_{h}^{\perp}||_{0,\Omega_{i}}+C|P_{h}\bm{w}_{h}^{\perp}|_{1,\Omega_{i}},
\end{align*}
which, together with \eqref{Q-h}, \eqref{Align_Embedding} 
     and \eqref{P-e}, yields
\begin{align}\label{Align_widetildebmv_estimate4}
    \begin{aligned}
    \sum_{i=1}^{N}|\widetilde{\bm{v}}|^{2}_{1,\Omega_{i}}
    &\leq C\sum_{i=1}^{N}|P_{h}\bm{w}_{h}^{\perp}|^{2}_{1,\Omega_{i}}+
    CH^{-2}\sum_{i=1}^{N}||P_{h}\bm{w}_{h}^{\perp}-Q_{H}P_{h}\bm{w}_{h}^{\perp}||^{2}_{0,\Omega_{i}}\\
    &\leq C|P_{h}\bm{w}_{h}^{\perp}|^{2}_{1}+CH^{-2}||P_{h}\bm{w}_{h}^{\perp}-Q_{H}P_{h}\bm{w}_{h}^{\perp}||_{0}^{2}\\
    &\leq C|P_{h}\bm{w}_{h}^{\perp}|^{2}_{1}
    \leq C||\bm{{\rm curl}}P_{h}\bm{w}_{h}^{\perp}||^{2}_{0}
    = C||\bm{{\rm curl}}\bm{w}_{h}^{\perp}||^{2}_{0}. 
    \end{aligned}
\end{align}

For the second term in \eqref{Align_widetildebmv_estimate3}, 
  by \eqref{Q-h}, \eqref{Align_Embedding} 
     and \eqref{P-e}, we have
\begin{align}\label{A-3} 
    \begin{aligned}
    \frac{C}{\delta H }||\widetilde{\bm{v}}||_{0}^{2}
    &=\frac{C}{\delta H}||P_{h}\bm{w}_{h}^{\perp}-Q_{H}P_{h}\bm{w}_{h}^{\perp}||_{0}^{2}\\
    &\leq \frac{C}{\delta H }CH^{2}|P_{h}\bm{w}_{h}^{\perp}|^{2}_{1}
    \leq C\frac{H}{\delta}||\bm{{\rm curl}}P_{h}\bm{w}_{h}^{\perp}||^{2}_{0}
    = C\frac{H}{\delta}||\bm{{\rm curl}}\bm{w}_{h}^{\perp}||^{2}_{0}. 
    \end{aligned}
\end{align}
Combining \eqref{Align_Local_Estimates_I1I2I3}, \eqref{Align_widetildebmv_estimate3}, \eqref{Align_widetildebmv_estimate4} and \eqref{A-3}, we get
\begin{align*}
I_{1}=C\sum_{i=1}^{N}\delta ^{-2}||\widetilde{\bm{v}}||_{0,\Omega_{i,\delta }}^{2}\leq C(1+\frac{H}{\delta})||\bm{{\rm curl}}\bm{w}_{h}^{\perp}||_{0}^{2}.
\end{align*}

For the second term $I_{2}$ in \eqref{Align_Local_Estimates_I1I2I3}, by Lemma \ref{Lemma_Ph}, we deduce
\begin{align*}
I_{2}&=C\sum_{i=1}^{N}\delta ^{-2}||\widetilde{\bm{v}}-\bm{v}||_{0,\Omega_{i,\delta }}^{2}
= C\delta^{-2}\sum_{i=1}^{N}||\bm{w}_{h}^{\perp}-P_{h}\bm{w}_{h}^{\perp}||_{0,\Omega_{i,\delta }}^{2}\\
&\leq C\delta^{-2}||\bm{w}_{h}^{\perp}-P_{h}\bm{w}_{h}^{\perp}||_{0}^{2}
\leq Ch^{2}\delta^{-2}||\bm{{\rm curl}}\bm{w}_{h}^{\perp}||_{0}^{2}\\
&\leq C||\bm{{\rm curl}}\bm{w}_{h}^{\perp}||_{0}^{2}.
\end{align*}
For the third term $I_{3}$ in \eqref{Align_Local_Estimates_I1I2I3}, we have,
  because of \eqref{Align_Coarse_Component_a1},
\begin{align*}
I_{3}=C||\bm{{\rm curl}}\bm{v}||_{0}^{2}\leq C\{||\bm{{\rm curl}}\bm{w}_{h}^{\perp}||_{0}^{2}+||\bm{{\rm curl}}\bm{w}_{0}||_{0}^{2}\}
\leq C||\bm{{\rm curl}}\bm{w}_{h}^{\perp}||_{0}^{2}.
\end{align*}
By using \eqref{Align_Local_Estimates_I1I2I3} and the estimates of three terms $I_{1},I_{2}$ and $I_{3}$, we obtain
\begin{align}\label{Align_Local_Estimates_curl}
\sum_{i=1}^{N}||\bm{{\rm curl}}\bm{w}_{i}||_{0,\Omega_{i}^{'}}^{2}\leq I_{1}+I_{2}+I_{3}\leq C(1+\frac{H}{\delta})
||\bm{{\rm curl}}\bm{w}_{h}^{\perp}||_{0}^{2}.
\end{align}
\item[(4)] For the terms $||\bm{w}_{i}||_{0,\Omega_{i}^{'}}$ in \eqref{Align_Stable_Decomposition_perp}, we have, by \eqref{d-u}, the fact $|\theta_i|\le 1$,
   triangle inequality,   finite overlapping and \eqref{Align_Coarse_Component_a2}, 
\begin{align}\label{Align_Local_Estimates_L2}
    \begin{aligned}
    \sum_{i=1}^{N}||\bm{w}_{i}||^{2}_{0,\Omega_{i}^{'}}
    &\leq C\sum_{i=1}^{N}||\theta_{i}(\bm{w}_{h}^{\perp}-\bm{w}_{0})||^{2}_{0,\Omega_{i}^{'}}
    \leq C\sum_{i=1}^{N}\{||\bm{w}_{h}^{\perp}||^{2}_{0,\Omega_{i}^{'}}+||\bm{w}_{0}||^{2}_{0,\Omega_{i}^{'}}\}\\
    &\leq C||\bm{w}_{h}^{\perp}||_{0}^{2}+C||\bm{w}_{0}||_{0}^{2}\leq C||\bm{w}_{h}^{\perp}||_{0}^{2}.
    \end{aligned}
\end{align}
\end{itemize}

Finally, combining \eqref{Align_Coarse_Component_a1}, \eqref{Align_Coarse_Component_a2} \eqref{Align_Local_Estimates_curl} and \eqref{Align_Local_Estimates_L2}, we complete the proof of \eqref{Align_Stable_Decomposition_perp}. 
By \eqref{d-u}, adding  \eqref{Align_Stable_Decomposition_kercurl} and  \eqref{Align_Stable_Decomposition_perp}, we get \eqref{upper},  noting that 
the decomposition in last equation of \eqref{d-u} is orthogonal under $a_{\bm{{\rm curl}}}(\cdot,\cdot)$. \qed
\end{proof}
\par Now we are in a position to give a proof of the main result.
\par\noindent{\bf Proof of Theorem \ref{Theorem_Mainresult}}:\ \ Based on Lemma 2.5 in Chapter 2 in \cite{MR2104179} and Theorem \ref{Theorem_Stable_Decomposition} above, we obtain
\begin{align*}
\frac{1}{C_{1}(1+\frac{H}{\delta})} a_{\bm{{\rm curl}}}(\bm{v}_{h},\bm{v}_{h})\leq a_{\bm{{\rm curl}}}(B_{h}^{-1}A_{h}\bm{v}_{h},\bm{v}_{h})\ \ \ \ \forall\ \bm{v}_{h}\in \mathcal{ND}_{h,0},
\end{align*}
which completes the proof of this theorem. \qed

\section{Extension to $\bm{H}({\rm div};\Omega)$}
\par In this section, we extend the theoretical techniques to the overlapping Schwarz method in $\bm{H}({\rm div};\Omega)$. For convenience of theoretical analysis, we first define some notations: 
\begin{align*}
\bm{H}_{0}({\rm div};\Omega):&=\{\ \bm{v}\in \bm{H}({\rm div};\Omega)\ |\ \bm{u}\cdot\bm{n}=0\ \},\\
\bm{H}_{0}^{\perp}({\rm div};\Omega):&=\bm{H}_{0}({\rm div};\Omega)\cap \bm{H}(\bm{{\rm curl}}_{0};\Omega),
\end{align*}
and we know that (see \cite{MR851383,MR1626990})
\begin{align}\label{div_embed}
\bm{H}_{0}^{\perp}({\rm div};\Omega)\hookrightarrow \bm{H}^{1}(\Omega).
\end{align}
It is known that the following $\bm{L}^{2}$-orthogonal (also $a_{{\rm div}}(\cdot,\cdot)$-orthogonal) decomposition holds (see \cite{MR2104179}):
\begin{align*}
\bm{H}_{0}({\rm div};\Omega)=\bm{{\rm curl}}\bm{H}_{0}(\bm{{\rm curl}};\Omega)\oplus \bm{H}_{0}^{\perp}({\rm div};\Omega)
=\bm{{\rm curl}}\bm{H}_{0}^{\perp}(\bm{{\rm curl}};\Omega)\oplus \bm{H}_{0}^{\perp}({\rm div};\Omega).
\end{align*}
\par Consider the model problem
\begin{equation}\notag
     \begin{cases}
        {\rm \nabla}{\rm div}\bm{u}+\bm{u}=\bm{f} \ \ &\text{in $\Omega,$}\\
         \ \ \ \ \ \bm{n}\cdot\bm{u}=0\ \ &\text{on $\partial\Omega.$}
     \end{cases}
\end{equation}
Its variational form is
\begin{equation}\notag 
    \begin{cases}
    \text{Given $\bm{f} \in \bm{L}^{2}(\Omega)$, find $\bm{u}\in \bm{H}_{0}({\rm div};\Omega)$ such that}\\
    \text{$a_{{\rm div}}(\bm{u},\bm{v})=(\bm{f},\bm{v})_{0}\ \ \ \ \forall\ \bm{v}\in \bm{H}_{0}({\rm div};\Omega)$},
    \end{cases}
\end{equation}
where 
\begin{align}\label{a-div}
a_{{\rm div}}(\bm{w},\bm{v})=\int_{\Omega}\big({\rm div}\bm{w}\ {\rm div}\bm{v}+\bm{w}\cdot\bm{v}\big)dx \ \ \ \ \forall\ \bm{w},\bm{v}\in \bm{H}({\rm div};\Omega). 
\end{align}
We consider $k$-th Raviart-Thomas finite element space
\begin{align*}
\mathcal{RT}_{h,0}:=\{\ \bm{v}\in \bm{H}_{0}({\rm div};\Omega)\ |\ \bm{v}|_{\tau}\in \mathcal{RT}_{k}(\tau)\ \ \forall\ \tau\in \mathcal{T}_{h}\},
\end{align*}
with $\mathcal{RT}_{k}(\tau)$ being the $k$-th order local Raviart-Thomas polynomial space on element $\tau$. It also admits the discrete Helmholtz decomposition:
\begin{align}\label{D-decompose-div}
\mathcal{RT}_{h,0}=\bm{{\rm curl}}\mathcal{ND}_{h,0}\oplus \mathcal{RT}_{h,0}^{\perp}=\bm{{\rm curl}}\mathcal{ND}_{h,0}^{\perp}\oplus \mathcal{RT}_{h,0}^{\perp},
\end{align}  
where 
\begin{align}\label{RT-perp}
\mathcal{RT}_{h,0}^{\perp}=\{\ \bm{w}_{h}\in \mathcal{RT}_{h,0}\ |\ (\bm{w}_{h},\bm{{\rm curl}}\bm{v}_{h})_{0}=0,\ \ \ \ \forall\ \bm{v}_{h}\in \mathcal{ND}_{h,0}\}.
\end{align}
\par As for the analysis in $\bm{H}(\bm{{\rm curl}};\Omega)$, we may define similarly an $\bm{L}^{2}$-orthogonal projector $\widetilde{\Theta}^{\perp}:\bm{H}_{0}({\rm div};\Omega)\to \bm{H}_{0}^{\perp}({\rm div};\Omega)$. Define $\widetilde{\bm{V}}^{+}:=\widetilde{\Theta}^{\perp}\mathcal{RT}_{h,0}^{\perp}\subset\bm{H}_{0}^{\perp}({\rm div};\Omega)$. Further, we define $\widetilde{P}_{h}:\bm{H}_{0}^{\perp}({\rm div};\Omega)\to \widetilde{\bm{V}}^{+}$ as
\begin{align}\label{P-h-div}
({\rm div}\widetilde{P}_{h}\bm{w},{\rm div}\bm{v})_{0}=({\rm div}\bm{w},{\rm div}\bm{v})_{0}\ \ \ \ \forall\ \bm{w}\in \bm{H}_{0}^{\perp}({\rm div};\Omega), \bm{v}\in \widetilde{\bm{V}}^{+}.
\end{align}
Since the Poincar\'e inequality holds in $\bm{H}_{0}^{\perp}({\rm div};\Omega)$, we know that $\widetilde{P}_{h}$ is well-defined. Moreover, we extend the operator $\widetilde{P}_{h}$ to $\bm{H}_{0}({\rm div};\Omega)$ by 
\begin{align}\label{P-h-div-1}
    \begin{aligned}
&\widetilde{P}_{h}{\bf curl}\bm{w}=\bm{0},\\
&\widetilde{P}_{h}\bm{v}=\widetilde{P}_{h}{\bf curl}\bm{w}+\widetilde{P}_{h}\bm{z}=\widetilde{P}_{h}\bm{z},
    \end{aligned}
\end{align}
where $\bm{v}={\bf curl}\bm{w}+\bm{z}\in{\bf curl}\bm{H}_{0}(\bm{{\rm curl}};\Omega)\oplus\bm{H}_{0}^{\perp}({\rm div};\Omega).$ 
Using the similar argument as in \eqref{P-T} and \eqref{P-e}, we get that for any $\bm{w}_{h}^{\perp}\in \mathcal{RT}_{h,0}^{\perp}$,
\begin{align}\label{P_T_div}
\widetilde{P}_{h}\bm{w}_{h}^{\perp}=\widetilde{\Theta}^{\perp} \bm{w}_{h}^{\perp},
\ \ \ \ \
{\rm div}\widetilde{P}_{h}\bm{w}_{h}^{\perp}={\rm div}\widetilde{\Theta}^{\perp}\bm{w}_{h}^{\perp}
={\rm div}\bm{w}_{h}^{\perp}.
\end{align}

\begin{lemma}\label{Lemma_Ph_div}
\cite{MR2104179} Let $\Omega$ be convex. Then
\begin{align*}
||\bm{w}_{h}^{\perp}-\widetilde{P}_{h}\bm{w}_{h}^{\perp}||_{0}\leq Ch||{\rm div}\bm{w}_{h}^{\perp}||_{0}\ \ \ \ \forall\ \bm{w}_{h}^{\perp}\in  \mathcal{RT}_{h,0}^{\perp},
\end{align*}
with the constant $C$ independent of $\bm{w}_{h}^{\perp}$ and $h$, where $\widetilde{P}_{h}$ is defined in \eqref{P-h-div}, \eqref{P-h-div-1} and $\mathcal{RT}_{h,0}^{\perp}$ is defined in \eqref{RT-perp}.
\end{lemma}

\begin{lemma}\label{Lemma_QH_div}
\cite{MR2104179} Let $\mathcal{T}_{H}$ be shape-regular and quasi-uniform and $\widetilde{Q}_{H}:\bm{L}^{2}(\Omega)\to \mathcal{RT}_{H,0}$ be a $\bm{L}^{2}$-orthogonal projector. Then for $\bm{u}\in \bm{H}^{1}(\Omega)$, we have
\begin{align}\label{Q-h-div}
    \begin{aligned}
    ||{\rm div}(\widetilde{Q}_{H}\bm{u})||_{0}&\leq C|\bm{u}|_{1},\\
    ||\bm{u}-\widetilde{Q}_{H}\bm{u}||_{0}&\leq CH|\bm{u}|_{1},
    \end{aligned}
\end{align}
with the constant $C$ independent of $\bm{u}$ and $H$.
\end{lemma}

\par Similar to the overlapping Schwarz method in $\bm{H}(\bm{{\rm curl}};\Omega)$ in Section 3, we have the notations $\Omega_{i}^{'},\ \Omega_{i},\ \Omega_{i}^{0}$ and $\Omega_{i,\delta}$. We denote by $\bm{W}_{i}:=\mathcal{RT}_{h,0}\cap \bm{H}_{0}({\rm div};\Omega_{i}^{'}) \ (i=1,2,...,N)$ local subspaces. As for the analysis in $\bm{H}(\bm{{\rm curl}};\Omega)$, we only need to prove following Theorem \ref{Theorem_Stable_Decomposition_divdiv}.

For any $\bm{u}_{h}\in \mathcal{RT}_{h,0}$, we decompose it as 
    $\bm{u}_{h}=  \sum_{i=0}^N \bm{u}_{i} $, where
\begin{align}\label{d-u-div} 
    \begin{aligned}
    \bm{u}_{0} &:={\bf curl}\bm{w}_{0}+\bm{z}_{0}\in \mathcal{RT}_{H,0},\\ 
    \bm{u}_{i} &:={\bf curl}\bm{w}_{i}+\bm{z}_{i}\in \bm{W}_{i},\ \ \ \ i=1,2,...,N,\\
    \bm{w}_{0} &= Q_H P_h \bm w_h^\perp, \\
    \bm{w}_{i} &= \Pi_{E,h} (\theta_i (\bm w_h^\perp - \bm w_0)),  \\
    \bm{z}_{0} &= \widetilde{Q}_{H}\widetilde{P}_{h}\bm{z}_{h}^{\perp},\\
    \bm{z}_{i} &= \Pi_{F,h}(\theta_{i}(\bm{z}_{h}^{\perp}-\bm{z}_{0})),\\
    \bm{u}_{h} &= {\bf curl}\bm{w}_{h}^{\perp}+\bm{z}_{h}^{\perp},
    \end{aligned} 
\end{align} 
where $\bm w_h^\perp$ and $\bm{z}_{h}^{\perp}$ are defined in \eqref{D-decompose-div},
      $\Pi_{F,h}$ is the H-div interpolation operator to $\mathcal{RT}_{h,0}$,
      $\theta_i$ is defined in \eqref{theta}, 
      $\widetilde{P}_h$ is defined in \eqref{P-h-div},
      $\widetilde{Q}_H$ is defined in \eqref{Q-h-div},
      $\Pi_{E,h}$ is the H-curl interpolation operator to $\mathcal{ND}_{h,0}$,
      $P_h$ is defined in \eqref{P-h} and
      $Q_H$ is defined in \eqref{Q-h}.

\begin{theorem}\label{Theorem_Stable_Decomposition_divdiv}
Let $\bm{u}_{h}\in \mathcal{RT}_{h,0}$ be decomposed in \eqref{d-u-div}.
It holds that 
\begin{align}\label{upper-div}
  \sum_{i=0}^{N}a_{{\rm div}}(\bm{u}_{i},\bm{u}_{i})
\leq C(1+\frac{H}{\delta})a_{{\rm div}}(\bm{u}_{h},\bm{u}_{h}). 
\end{align}
\end{theorem}
\begin{proof}
By the discrete Helmholtz decomposition in the last equation in \eqref{d-u-div},
we decompose $\bm{u}_{h}$ in \eqref{d-u-div} as 
\begin{align*} 
\bm{u}_{h}={\bf curl}\bm{w}_{0}+\bm{z}_{0}+\sum_{i=1}^N ({\bf curl}\bm{w}_{i}+\bm{z}_{i}).
\end{align*}

\par For the terms $\bm{{\rm curl}}\bm{w}_{i}$ of $\bm u_h$, by Section 3 in this paper, we have $\bm{w}_{0}\in \mathcal{ND}_{H,0}$ and $\bm{w}_{i}\in \bm{V}_{i}$, which satisfy 
\begin{align}\label{Align_Stable_Decomposition_kerdiv}
\sum_{i=0}^{N}a_{{\rm div}}(\bm{{\rm curl}}\bm{w}_{i},\bm{{\rm curl}}\bm{w}_{i})
\leq C(1+\frac{H}{\delta})a_{{\rm div}}(\bm{{\rm curl}}\bm{w}_{h}^{\perp},\bm{{\rm curl}}\bm{w}_{h}^{\perp}). 
\end{align}

For the terms $\bm{z}_{i}$,  we will prove the following by the steps (1) and (2) below. 
\begin{align}\label{Align_Stable_Decomposition_perp_div}
\sum_{i=0}^{N}a_{{\rm div}}(\bm{z}_{i},\bm{z}_{i})
\leq C(1+\frac{H}{\delta})a_{{\rm div}}(\bm{z}_{h}^{\perp},\bm{z}_{h}^{\perp}).
\end{align}

\begin{itemize}
\item[(1)] For the coarse component $\bm{z}_{0}$, by \eqref{d-u-div}, \eqref{Q-h-div}, \eqref{div_embed} and \eqref{P_T_div}, we have
\begin{align*}
||{\rm div}\bm{z}_{0}||_{0}
=||{\rm div}\widetilde{Q}_{H}\widetilde{P}_{h}\bm{z}_{h}^{\perp}||_{0}
\leq C|\widetilde{P}_{h}\bm{z}_{h}^{\perp}|_{1}
\leq C||{\rm div}\widetilde{P}_{h}\bm{z}_{h}^{\perp}||_{0}
=C||{\rm div}\bm{z}_{h}^{\perp}||_{0}.
\end{align*}
By \eqref{d-u-div}, \eqref{Q-h-div}, \eqref{P_T_div} and the definition $\widetilde{\Theta}^{\perp}$, we get
\begin{align*}
||\bm{z}_{0}||_{0}
=||\widetilde{Q}_{H}\widetilde{P}_{h}\bm{z}_{h}^{\perp}||_{0}
\leq ||\widetilde{P}_{h}\bm{z}_{h}^{\perp}||_{0}
=||\widetilde{\Theta}^{\perp}\bm{z}_{h}^{\perp}||_{0}
\leq ||\bm{z}_{h}^{\perp}||_{0}.
\end{align*}
The above two inequalities imply that, by the definition of $a_{{\rm div}}(\cdot,\cdot)$ in \eqref{a-div},
\begin{align}\label{Align_Coarse_Component_a_div}
a_{{\rm div}}(\bm{z}_{0},\bm{z}_{0})\leq Ca_{{\rm div}}(\bm{z}_{h}^{\perp},\bm{z}_{h}^{\perp}).
\end{align}
\item[(2)] For the local components $\{\bm{z}_{i}\}_{i=1}^{N}$, by the properties of $\{\theta_{i}\}_{i=1}^{N}$ and $\Pi_{F,h}$ (see Lemma 10.13 in \cite{MR2104179}), we have
    \begin{align*}
    \sum_{i=1}^{N}||{\rm div}\bm{z}_{i}||_{0,\Omega_{i}^{'}}^{2}\leq C\sum_{i=1}^{N}\delta ^{-2}||\bm{z}||_{0,\Omega_{i,\delta }}^{2}+C||{\rm div}\bm{z}||_{0}^{2},
    \end{align*}
     where 
    \begin{align}\notag
    \bm{z}=\bm{z}_{h}^{\perp}-\bm{z}_{0}=\bm{z}_{h}^{\perp}-\widetilde{Q}_{H}\widetilde{P}_{h}\bm{z}_{h}^{\perp}.
    \end{align}
     Denoting by $\widetilde{\bm{z}}:=\widetilde{P}_{h}\bm{z}_{h}^{\perp}-\widetilde{Q}_{H}\widetilde{P}_{h}\bm{z}_{h}^{\perp}$ and using the triangle inequality, we obtain
    \begin{align}\label{Align_Local_Estimates_J1J2J3}
        \begin{aligned}
        \sum_{i=1}^{N}||{\rm div}\bm{z}_{i}||_{0,\Omega_{i}^{'}}^{2}
        &\leq C\sum_{i=1}^{N}\delta ^{-2}||\widetilde{\bm{z}}||_{0,\Omega_{i,\delta }}^{2}
             +C\sum_{i=1}^{N}\delta ^{-2}||\widetilde{\bm{z}}-\bm{z}||_{0,\Omega_{i,\delta }}^{2}
             +C||{\rm div}\bm{z}||_{0}^{2}\\
        &:=J_{1}+J_{2}+J_{3}.
        \end{aligned}
    \end{align}
    For the first term $J_{1}$ in \eqref{Align_Local_Estimates_J1J2J3}, by Lemma \ref{Lemma_Inequality_of_InnerStri} and the fact that $\widetilde{\bm{z}}\in \bm{H}^{1}(\Omega_{i})$, we get, because of finite overlapping,
    \begin{align}\label{Align_widetildebmz_estimate3}
    J_{1}=C\sum_{i=1}^{N}\delta ^{-2}||\widetilde{\bm{z}}||_{0,\Omega_{i,\delta }}^{2}
    &\leq C(1+\frac{H}{\delta})\sum_{i=1}^{N}|\widetilde{\bm{z}}|^{2}_{1,\Omega_{i}}+ 
    C\frac{1}{\delta H}||\widetilde{\bm{z}}||^{2}_{0,\Omega}. 
    \end{align}
    For the first term in \eqref{Align_widetildebmz_estimate3}, by the triangle inequality, inverse estimate and \eqref{Align_Q0i_Approximation}, we obtain
\begin{align*}
|\widetilde{\bm{z}}|_{1,\Omega_{i}}&=|\widetilde{P}_{h}\bm{z}_{h}^{\perp}-\widetilde{Q}_{H}\widetilde{P}_{h}\bm{z}_{h}^{\perp}|_{1,\Omega_{i}}
\leq |\widetilde{P}_{h}\bm{z}_{h}^{\perp}|_{1,\Omega_{i}}
+|\widetilde{Q}_{H}\widetilde{P}_{h}\bm{z}_{h}^{\perp}-Q_{0,\Omega_{i}}\widetilde{P}_{h}\bm{z}_{h}^{\perp}|_{1,\Omega_{i}}\\
&\leq |\widetilde{P}_{h}\bm{z}_{h}^{\perp}|_{1,\Omega_{i}}+
CH^{-1}||\widetilde{Q}_{H}\widetilde{P}_{h}\bm{z}_{h}^{\perp}-Q_{0,\Omega_{i}}\widetilde{P}_{h}\bm{z}_{h}^{\perp}||_{0,\Omega_{i}}\\
&\leq |\widetilde{P}_{h}\bm{z}_{h}^{\perp}|_{1,\Omega_{i}}
+CH^{-1}||\widetilde{Q}_{H}\widetilde{P}_{h}\bm{z}_{h}^{\perp}-\widetilde{P}_{h}\bm{z}_{h}^{\perp}||_{0,\Omega_{i}}\\
&\ \ \ \ +CH^{-1}||\widetilde{P}_{h}\bm{z}_{h}^{\perp}-Q_{0,\Omega_{i}}\widetilde{P}_{h}\bm{z}_{h}^{\perp}||_{0,\Omega_{i}}\\
&\leq |\widetilde{P}_{h}\bm{z}_{h}^{\perp}|_{1,\Omega_{i}}
+CH^{-1}||\widetilde{Q}_{H}\widetilde{P}_{h}\bm{z}_{h}^{\perp}-\widetilde{P}_{h}\bm{z}_{h}^{\perp}||_{0,\Omega_{i}}
+C|\widetilde{P}_{h}\bm{z}_{h}^{\perp}|_{1,\Omega_{i}},
\end{align*}
which, together with Lemma \ref{Lemma_QH_div}, \eqref{div_embed} and \eqref{P_T_div}, yields
\begin{align}\label{Align_widetildebmz_estimate4}
    \begin{aligned}
    \sum_{i=1}^{N}|\widetilde{\bm{z}}|^{2}_{1,\Omega_{i}}
    &\leq C\sum_{i=1}^{N}|\widetilde{P}_{h}\bm{z}_{h}^{\perp}|^{2}_{1,\Omega_{i}}+
    CH^{-2}\sum_{i=1}^{N}||\widetilde{P}_{h}\bm{z}_{h}^{\perp}-\widetilde{Q}_{H}\widetilde{P}_{h}\bm{z}_{h}^{\perp}||^{2}_{0,\Omega_{i}}\\
    &\leq C|\widetilde{P}_{h}\bm{z}_{h}^{\perp}|^{2}_{1}+
    CH^{-2}||\widetilde{P}_{h}\bm{z}_{h}^{\perp}-\widetilde{Q}_{H}\widetilde{P}_{h}\bm{z}_{h}^{\perp}||_{0}^{2}\\
    &\leq C|\widetilde{P}_{h}\bm{z}_{h}^{\perp}|^{2}_{1}
    \leq C||{\rm div}\widetilde{P}_{h}\bm{z}_{h}^{\perp}||^{2}_{0}
    = C||{\rm div}\bm{z}_{h}^{\perp}||^{2}_{0}. 
\end{aligned}
\end{align}
For the second term in \eqref{Align_widetildebmz_estimate3}, by Lemma \ref{Lemma_QH_div}, \eqref{div_embed} and \eqref{P_T_div}, we get
\begin{align*}
\frac{C}{\delta H}||\widetilde{\bm{z}}||_{0}^{2}
&=\frac{C}{\delta H}||\widetilde{P}_{h}\bm{z}_{h}^{\perp}-\widetilde{Q}_{H}\widetilde{P}_{h}\bm{z}_{h}^{\perp}||_{0}^{2}\\
&\leq \frac{C}{\delta H}CH^{2}|\widetilde{P}_{h}\bm{z}_{h}^{\perp}|^{2}_{1}
\leq C\frac{H}{\delta}||{\rm div}\widetilde{P}_{h}\bm{z}_{h}^{\perp}||^{2}_{0}
= C\frac{H}{\delta}||{\rm div}\bm{z}_{h}^{\perp}||^{2}_{0},
\end{align*}
which, together with \eqref{Align_widetildebmz_estimate3} and \eqref{Align_widetildebmz_estimate4}, yields
\begin{align*}
J_{1}=C\sum_{i=1}^{N}\delta ^{-2}||\widetilde{\bm{z}}||_{0,\Omega_{i,\delta }}^{2}
\leq C(1+\frac{H}{\delta})||{\rm div}\bm{z}_{h}^{\perp}||_{0}^{2}.
\end{align*}

For the second term $J_{2}$ in \eqref{Align_Local_Estimates_J1J2J3}, by Lemma \ref{Lemma_Ph_div}, we deduce
\begin{align*}
J_{2}
&=C\sum_{i=1}^{N}\delta ^{-2}||\widetilde{\bm{z}}-\bm{z}||_{0,\Omega_{i,\delta }}^{2}
=C\delta^{-2}\sum_{i=1}^{N}||\bm{z}_{h}^{\perp}-\widetilde{P}_{h}\bm{z}_{h}^{\perp}||_{0,\Omega_{i,\delta }}^{2}\\
&\leq C\delta^{-2}||\bm{z}_{h}^{\perp}-\widetilde{P}_{h}\bm{z}_{h}^{\perp}||_{0}^{2}
\leq Ch^{2}\delta^{-2}||{\rm div}\bm{z}_{h}^{\perp}||_{0}^{2}\\
&\leq C||{\rm div}\bm{z}_{h}^{\perp}||_{0}^{2}.
\end{align*}
For the third term $J_{3}$ in \eqref{Align_Local_Estimates_J1J2J3}, we have, because of \eqref{Align_Coarse_Component_a_div},
\begin{align*}
J_{3}=C||{\rm div}\bm{z}||_{0}^{2}\leq 2\{||{\rm div}\bm{z}_{h}^{\perp}||_{0}^{2}+||{\rm div}\bm{z}_{0}||_{0}^{2}\}
\leq C||{\rm div}\bm{z}_{h}^{\perp}||_{0}^{2}.
\end{align*}
By using \eqref{Align_Local_Estimates_J1J2J3} and the estimates of three terms $J_{1},\ J_{2}$ and $J_{3}$, we obtain
\begin{align}\label{Align_Local_Estimates_div}
\sum_{i=1}^{N}||{\rm div}\bm{z}_{i}||_{0,\Omega_{i}^{'}}^{2}\leq J_{1}+J_{2}+J_{3}\leq C(1+\frac{H}{\delta})||{\rm div}\bm{z}_{h}^{\perp}||_{0}^{2}.
\end{align}
\par For the $\bm{L}^{2}$-norm estimate, we have, by \eqref{d-u-div}, the fact $|\theta_{i}|\leq 1$, finite overlapping and \eqref{Align_Coarse_Component_a_div},
\begin{align}\label{Align_Local_Estimates_L2_div}
    \begin{aligned}
    \sum_{i=1}^{N}||\bm{z}_{i}||^{2}_{0,\Omega_{i}^{'}}
    &\leq C\sum_{i=1}^{N}||\theta_{i}(\bm{z}_{h}^{\perp}-\bm{z}_{0})||^{2}_{0,\Omega_{i}^{'}}
    \leq C\sum_{i=1}^{N}\{||\bm{z}_{h}^{\perp}||^{2}_{0,\Omega_{i}^{'}}+||\bm{z}_{0}||^{2}_{0,\Omega_{i}^{'}}\}\\
    &\leq C||\bm{z}_{h}^{\perp}||_{0}^{2}+C||\bm{z}_{0}||_{0}^{2}
    \leq C||\bm{z}_{h}^{\perp}||_{0}^{2}.
    \end{aligned}
\end{align}
\end{itemize}
Combining \eqref{Align_Coarse_Component_a_div}, \eqref{Align_Local_Estimates_div} and \eqref{Align_Local_Estimates_L2_div}, we completes the proof of \eqref{Align_Stable_Decomposition_perp_div}. Finally, by \eqref{d-u-div}, \eqref{Align_Stable_Decomposition_kerdiv} and \eqref{Align_Stable_Decomposition_perp_div}, we get \eqref{upper-div},  noting that 
the decomposition \eqref{D-decompose-div} is also orthogonal under $a_{{\rm div}}(\cdot,\cdot)$. \qed
\end{proof}

\section{Conclusions}

\par In this paper, we prove that the estimates of the condition numbers of the overlapping Schwarz methods in $\bm{H}(\bm{{\rm curl}};\Omega)$ and $\bm{H}({\rm div};\Omega)$ are bounded by $C\left(1+\frac{H}{\delta}\right)$, which is similar as the case in $H^{1}(\Omega)$. We emphasize that the previous bound is $C\left(1+\frac{H^{2}}{\delta^{2}}\right)$. So we close this open problem for overlapping Schwarz methods in $\bm{H}(\bm{{\rm curl}};\Omega)$ and $\bm{H}({\rm div};\Omega)$.

\begin{small}
\bibliographystyle{plain}
\bibliography{reference}

\begin{thebibliography}{10}

\bibitem{MR1626990}
Ch\'{e}rif Amrouche, Christine Bernardi, Monique Dauge, and Vivette Girault.
\newblock Vector potentials in three-dimensional non-smooth domains.
\newblock {\em Mathematical Methods in the Applied Sciences}, 21(9):823--864,
  1998.

\bibitem{MR3985469}
Marcella Bonazzoli, Victorita Dolean, Ivan Graham, Euan Spence, and Pierre
  Tournier.
\newblock Domain decomposition preconditioning for the high-frequency
  time-harmonic {M}axwell equations with absorption.
\newblock {\em Mathematics of Computation}, 88(320):2559--2604, 2019.

\bibitem{MR1066830}
James Bramble and Jinchao Xu.
\newblock Some estimates for a weighted {$L^2$} projection.
\newblock {\em Mathematics of Computation}, 56(194):463--476, 1991.

\bibitem{MR1376107}
Susanne Brenner.
\newblock A two-level additive {S}chwarz preconditioner for nonconforming plate
  elements.
\newblock {\em Numerische Mathematik}, 72(4):419--447, 1996.

\bibitem{MR1348039}
Susanne Brenner.
\newblock Two-level additive {S}chwarz preconditioners for nonconforming finite
  element methods.
\newblock {\em Mathematics of Computation}, 65(215):897--921, 1996.

\bibitem{MR1771050}
Susanne Brenner.
\newblock Lower bounds for two-level additive {S}chwarz preconditioners with
  small overlap.
\newblock {\em SIAM Journal on Scientific Computing}, 21(5):1657--1669, 2000.

\bibitem{MRDryjaWidlund}
Maksymilian Dryja and Olaf Widlund.
\newblock An additive variant of the {S}chwarz alternating method for the case
  of many subregions, {T}echnical {R}eport {T}{R}-339, also {U}ltracomputer
  {N}ote 131, {D}epartment of {C}omputer {S}cience, {C}ourant {I}nstitute.
\newblock 1987.

\bibitem{MR1273155}
Maksymilian Dryja and Olof Widlund.
\newblock Domain decomposition algorithms with small overlap.
\newblock {\em SIAM Journal on Scientific Computing}, 15(3):604--620, 1994.

\bibitem{MR851383}
Vivette Girault and Pierre-Arnaud Raviart.
\newblock {\em Finite {E}lement {M}ethods for {N}avier-{S}tokes {E}quations},
  volume~5 of {\em Springer Series in Computational Mathematics}.
\newblock Springer-Verlag, Berlin, 1986.
\newblock Theory and algorithms.

\bibitem{MR3647952}
Ivan Graham, Euan Spence, and Eero Vainikko.
\newblock Domain decomposition preconditioning for high-frequency {H}elmholtz
  problems with absorption.
\newblock {\em Mathematics of Computation}, 86(307):2089--2127, 2017.

\bibitem{MR2009375}
Ralf Hiptmair.
\newblock Finite elements in computational electromagnetism.
\newblock {\em Acta Numerica}, 11:237--339, 2002.

\bibitem{MR1838270}
Ralf Hiptmair and Andrea Toselli.
\newblock Overlapping and multilevel {S}chwarz methods for vector valued
  elliptic problems in three dimensions.
\newblock In {\em Parallel solution of partial differential equations
  ({M}inneapolis, {MN}, 1997)}, volume 120 of {\em The IMA Volumes in
  Mathematics and its Applications}, pages 181--208. Springer, New York, 2000.

\bibitem{MR4379971}
Qigang Liang and Xuejun Xu.
\newblock A two-level preconditioned {H}elmholtz-{J}acobi-{D}avidson method for
  the {M}axwell eigenvalue problem.
\newblock {\em Mathematics of Computation}, 91(334):623--657, 2022.

\bibitem{MR592160}
Jean-Claude N{\'e}d{\'e}lec.
\newblock Mixed finite elements in {${\bf R}^{3}$}.
\newblock {\em Numerische Mathematik}, 35(3):315--341, 1980.

\bibitem{MR864305}
Jean-Claude N{\'e}d{\'e}lec.
\newblock A new family of mixed finite elements in {${\bf R}^3$}.
\newblock {\em Numerische Mathematik}, 50(1):57--81, 1986.

\bibitem{MR1011446}
Ridgway Scott and Shangyou Zhang.
\newblock Finite element interpolation of nonsmooth functions satisfying
  boundary conditions.
\newblock {\em Mathematics of Computation}, 54(190):483--493, 1990.

\bibitem{MR1794350}
Andrea Toselli.
\newblock Overlapping {S}chwarz methods for {M}axwell's equations in three
  dimensions.
\newblock {\em Numerische Mathematik}, 86(4):733--752, 2000.

\bibitem{MR2104179}
Andrea Toselli and Olof Widlund.
\newblock {\em Domain {D}ecomposition {M}ethods---{A}lgorithms and {T}heory},
  volume~34 of {\em Springer Series in Computational Mathematics}.
\newblock Springer-Verlag, Berlin, 2005.

\end{thebibliography}
\end{small}

\end{document}